\documentclass[12pt]{amsart}
\usepackage{amscd}
\usepackage{amssymb}

\newcommand{\bA}{{\mathbb A}}
\newcommand{\bC}{{\mathbb C}}

\newcommand{\bQ}{{\mathbb Q}}

\newcommand{\bZ}{{\mathbb Z}}
\newcommand{\bG}{{\mathbb G}}

\newcommand{\la}{{\langle}}
\newcommand{\ra}{{\rangle}}

\newtheorem{thm}{Theorem}[section]
\newtheorem{lemma}[thm]{Lemma}
\newtheorem{cor}[thm]{Corollary}
\newtheorem{prop}[thm]{Proposition}

\numberwithin{equation}{section}

\begin{document}

\title[Chow ring and Weyl group invariants]{Note on restriction maps of Chow rings to  Weyl group invariants}
 
\author{Nobuaki Yagita}

\address{ faculty of Education, 
Ibaraki University,
Mito, Ibaraki, Japan}
 
\email{nobuaki.yagita.math@vc.ibaraki.ac.jp, }


\maketitle

\begin{abstract}
Let $G$ be an algebraic group over $\bC$ corresponding a compact simply connected Lie group.
When $H^*(G)$ has $p$-torsion, we see
$\rho^*_{CH}: CH^*(BG)\to CH^*(BT)^{W_G(T)}$
is always not surjective.  We also study the algebraic cobordism version $\rho^*_{\Omega}$.
In particular when $G=Spin(7)$ and $p=2$, we see each
Griffiths element in $CH^*(BG)$ is detected by an element
in $\Omega^*(BT)$,  
\end{abstract}

\section{Introduction}
Let $p$ be a prime number.
Let $G$ be a compact Lie group
and $T$ the maximal torus. 
Let  us write $H^*(-)=H^*(-;\bZ_{(p)})$, and  $BG,BT$  classifying spaces of $G,T$. 
  Let $W=W_G(T)=N_G(T)/T$ be the Weyl group
and $Tor\subset H^*(BG)$
be the ideal generated by torsion elements.
Then we have the restriction map 
\[ \rho^*_{H}: H^*(BG)\to H^*(BG)/Tor \subset H^*(BT)^W\]
by using the Becker-Gottlieb transfer.

It is well known that when $H^*(G)$ is $p$-torsion free
(hence $H^*(BG)$ is $p$-torsion free), then $\rho^*_H$ surjective.  However when $H^*(G)$ has $p$-torsion,
there are cases that $\rho^*_H$ are not surjective,
which are founded  by
Feshbach [Fe].

Let us write by $G_{\bC}, T_{\bC}$ the 
reductive group over $\bC$
and its maximal torus corresponding the Lie group $G,T$.
Let us write simply $CH^*(BG)=CH^*(BG_{\bC})_{(p)}$,
$CH^*(BT)=CH^*(BT_{\bC})_{(p)}$ the Chow rings
of $BG_{\bC}$ and $BT_{\bC}$
localized at $p$. 
We consider the Chow ring version
\[ \rho_{CH}^* : CH^*(BG)\to  CH^*(BG)/Tor \subset 
 CH^*(BT)^{W}\cong H^*(BT)^W.\]
 Our first observation is
\begin{thm} Let $G$ be simply connected.
If $H^*(G)$ has $p$-torsion, then the map
$\rho^*_{CH}$ is always not surjective.
\end{thm}
In the proof, we use an element $x\in H^4(BG)$ with
$\rho_{H}(x)\not \in Im(\rho_{CH}^*)$.  
Hence $x\not \in Im(cl)$ for the cycle map $cl:CH^*(BG)\to H^*(BG)$. The corresponding element $1\otimes x\in CH^*(B\bG_m\times BG)$ is the element founded as a counterexample for the integral Hodge and hence the integral Tate conjecture in [Pi-Ya].

Next, we consider elements in $Tor$.
To study torsion elements, we consider the following restriction map
\[ res_H : H^*(BG) \to \Pi_{A:abelian \subset G} H^*(BA)^{W_G(A)}. \]
  There are cases such that $res_{H}$ is not injective, while for many cases $res_H$ are injective.
We consider the Chow ring version ([To1,2]) of the above restriction map
\[ res_{CH} : CH^*(BG) \to \Pi_{A:ab.} CH^*(BA)^{W_G(A)}
\subset \Pi_{A:ab.}H^*(BA)^{W_G(A)}. \]
In general $res_{CH}$ has non zero kernel.
In particular, elements in $Ker(cl)$ (i.e. Griffiths elements)
for the cycle map $cl: CH^*(BG)\to H^*(BG)$ are always
in $Ker(res_{CH})$. Namely Griffiths elements are not detected
by $res_{CH}$.

On the other hand, if the Totaro conjecture
\[ CH^*(BG)\cong BP^*(BG)\otimes_{BP^*}\bZ_{(p)}\]
(for the Brown-Peterson cohomology $BP^*(-)$)
is correct,  then of course all elements in $CH^*(BG)$ are detected by elements in $BP^*(BG)$.
We show that there is a case that Griffiths elements are detected by $\rho^*_{\Omega}$ the restriction for
algebraic cobordism theory $\Omega^*(-)$.

Let $\Omega^*(X)=MGL^{2*,*}(X)\otimes _{MU_{(p)}^*}BP^*$ be the $BP$-version of the algebraic cobordism defined by Voevodsky, Levine-Morel ([Vo1], [Le-Mo1,2] such that
$ CH^*(X)\cong \Omega^*(X)\otimes_{BP^*}\bZ_{(p)}.$
In particular, we consider the case $G=Spin(7)$ and $p=2$.
%The number of  conjugacy classes 
%of maximal abelian subgroup is two ; one is $T$ and the other 
%is $A'\cong (\bZ/2)^4$.  
We note that there are (non zero)
Griffiths elements in $CH^*(BG)$.
\begin{thm}  Let $G=Spin(7)$ and $p=2$. 
% Then
%\[ \rho_{\Omega}^* :\Omega^*(BG)
%\subset  \Omega^*(BA')^{W_G(A')}\times \Omega^*(BT)^{W}\]
%is injective.  In particular, 
Then each Griffiths element (in $CH^*(BG)$)
is detected
by an element in $\Omega^*(BT)^W\cong BP^*(BT)^W$.
\end{thm}

In $\S 2$ we study the map $\rho^*_{H}$ for the ordinary
cohomology theory, and recall  Feshbach's result.
In $\S 3$, we study the Chow ring version and show Theorem1.1.
In $\S 4$, we study the case $G=Spin(n)$.
In $\S 5$, we study the $BP^*$-version and 
the algebraic cobordism version for the restriction $\rho^*$.
In $\S 6$, we write down the case $G=Spin(7)$
quite explicitly, and show Theorem 1.2.
In the last section, we note some partial results for
the exceptional group $G=F_4$ and $p=3$.

The author thanks Kirill Zainoulline to start considering  this problem,
and Masaki Kameko  who gives many comments and suggestions,  and lets the author know 
works by Benson-Wood and Feshbach.

\section{cohomology theory and Feshbach theorem}

Let $p$ be a prime number.
Let $G$ be a compact Lie group
and $T$ the maximal Torus.
Then we have the restriction map
\[ \rho^*_H: H^*(BG)\to H^*(BT)^W\]
where $H^*(-)=H^*(-;\bZ_{(p)})$, $BG,BT$ are classifying spaces
and $W=W_G(T)=N_G(T)/T$ is the Weyl group.

It is well known that when $H^*(G)$ is $p$-torsion free, then $\rho^*_H$ is surjective (and hence is isomorphic).  However when $H^*(G)$ has $p$-torsion,
there are cases that $\rho^*_H$ are not surjective by
Feshbach. 

For a connected  compact Lie group $G$, we have the Becker-Gottlieb transfer
$\tau:H^*(BT)\to H^*(BG)$ such that $\tau\rho_H^*=\chi(G/T)$
for the Euler number $\chi(-)$, and $\rho_H^*\tau(x)=\chi(G/T)x$ for $x\in H^*(BT)^W$.  Let $\chi(G/T)=N$ and $Tor$ be the ideal of
$H^*(BG)$ generated by torsion elements.  Then we have the
injections 
\[ N\cdot H^*(BT)^W \subset H^*(BG)/Tor \subset H^*(BT)^W.\]
Feshbach found good criterion to see $\rho^*_H$ is surjecive.
\begin{thm} (Feshbach [Fe]) The restriction $\rho_H^*$
is surjective if and only if $(H^*(BG)/Tor)\otimes \bZ/p$ has not  nonzero nilpotent elements.
\end{thm}
\begin{proof} (Feshbach)
First note that $H^*(BT)\cong \bZ_{(p)}[t_1,...,t_{\ell}]$
for $|t_i|=2$.  Hence if $x^m=px'$ in $H^*(BT)$, then $x=px''$ for $x''\in H^*(BT)$.  Moreover if $x\in H^*(BT)^W$, then so
is $x'$ since $H^*(BT)$ is $p$-torsion free. Thus we see $H^*(BT)^W\otimes \bZ/p$ has no non zero nilpotent elements.

Assume that $\rho_H^*$ is not surjective, and
$x\in H^*(BT)^W$ but  
$x\not \in Im(\rho^*_H)$.  Let $s\ge 1$ be the smallest number such that $p^sx=\rho^*_H(y)$ for some $y\in H^*(BG)$.
Hence  $y\not =0\ mod(p)$.  Then
\[ \rho_H^*(y^N)=(p^sx)^N=p^{sN}x^N\in pN\cdot H^*(BT)^W\subset pIm(\rho^*_H).\]
This means $y$ is a nonzero nilpotent element in
$(H^*(BG)/Tor)\otimes \bZ/p$.
\end{proof}

Using this theorem, Feshbach [Fe] showed $\rho_H^*$ is surjective
for $G=G_2$, $Spin(n)$ for $n\le 10$, and is not surjective
for $Spin(11), Spin(12)$.  Wood [Wo] showed that $Spin(13)$
is not surjective but $Spin(n)$ for $14\le n\le 18$ are surjective.  Benson and Wood solved this problem completely,
namely $\rho^*$ is not surjective if and only if
$n\ge 11$ and $n=3,4,5\ mod(8)$.  

For odd prime,  we consider $mod(p)$ version
\[ \rho_{H/p}: H^*(BG;\bZ/p)\to H^*(BT;\bZ/p)^W\cong 
(H^*(BT)/p)^W.\] 
It is known that $\rho^*_{H/p}$ are surjective when $G=F_4$ for $p=3$ by Toda [Tod] using a completely different 
arguments.
Also different arguments (but without computations of $H^*(BT)^W$ for concrete cases), Kameko and Mimura [Ka-Mi] prove that 
$\rho^*_{H/p}$ are surjective when $G=E_6
,E_7$ for $p=3$ and $G=E_8$ for $p=5$.
(The only remain case is $G=E_8$, $p=3$ for odd primes.)

Kameko-Mimura get more strong result.
Recall the Milnor $Q_i$ operation
\[ Q_i: H^*(X;\bZ/p)\to H^{*+2p^i-1}(X;\bZ/p)\]
defined by $Q_0=\beta$ and $Q_{i+1}=[P^{p^i}Q_i,Q_iP^{p^i}]$
for the Bockstein $\beta$ and the reduced powers $P^j$.
\begin{thm} (Kameko-Mimura [Ka-Mi]) Let $G=F_4,E_6.E_7$ for
$p=3$ or $E_8$ for $p=5$.  Let us write a generator
by $x_4$ in $H^4(BG)\cong \bZ_{(p)}$.  Then we have 
 \[H^*(BT;\bZ/p)^W\cong H^{even}(BG;\bZ/p)/(Q_1Q_2x_4).\]
\end{thm}
\begin{cor} For $(G,p)$ in the above theorem,
$\rho_H^*$ is surjective.
\end{cor}
Since its $Q_0$-image is zero, we can identify
the element $Q_1Q_2(x_4)$ is in $H^*(BG)$ and $p$-torsion.
The above corollary is immediate from the following lemma.
\begin{lemma} If the composition
\[    (H^*(BG)/Tor)\otimes \bZ/p  \to
    H^*(BT)^W/p \to H^*(BT;\bZ/p)^W\]
is injective,  then $\rho_{H}^*$ is surjective.
\end{lemma}
\begin{proof}
Let $\rho^*_H$ be not surjecive and $y\in H^*(BT)^W$
with $y\not \in Im(\rho_H^*)$.  Then $p^sy=\rho^*_H(x)$ for $s\ge 1$ and an
additive generator $x\in H^*(BG)/Tor$.  Of course
$x=0 \in (H^*(BT)/p)^W.$
\end{proof}

\section{Chow rings}

 Let us write by $G_{\bC}, T_{\bC}$ the 
reductive group over $\bC$
and its maximal torus corresponding the Lie group $G$.
Let $CH^*(BG)=CH^*(BG_{\bC})_{(p}$ be the Chow ring
of $BG_{\bC}$ localized at $p$. 

The arguments of Feshbach also work for Chow rings
since the Becker-Gottlieb transfer is constructed by Totaro [To2].
\begin{lemma}
The restriction map $\rho_{CH}^* $ of Chow rings
is surjective if and only if $(CH^*(BG)/T)\otimes \bZ/p$ has not
nonzero nilpotent elements.
\end{lemma}
However if $H^*(G)$ has $p$-torsion and $G$ is simply connected, then
$(CH^*(BG)/Tor)\otimes \bZ/p$ always has non zero nilpotent elements.  In fact,  $c_2=px_4\in CH^4(BG)$ in the proof of Theorem 3.3 below,
is nilpotent in $(CH^*(BG)/(Tor))\otimes \bZ/p$.
However from the proof of the above lemma, we note 
\begin{cor}
If $x\in CH^*(BT)^W$ but $x\not \in Im(\rho_{CH}^*)$,
then there is $y\in CH^*(BG)$ such that
$\rho_{CH}^*(y)=p^sx$ for some $s\ge 1$ and
$y$ is non zero
nilpotent element in $(CH^*(BG)/(Tor))\otimes \bZ/p$.
\end{cor}
  
 Voevodsky [Vo1,2] defined the Milnor operation $Q_i$
on the mod $p$ motivic cohomology (over a perfect field $k$ of any $ch(k)$)
\[ Q_i: H^{*,*'}(X;\bZ/p)\to H^{*+2(p^n-1),*'+p^n-1}(X;\bZ/p)\]
which is  compatible with the usual topological $Q_i$ by
the realization map $t_{\bC}:H^{*,*'}(X;\bZ/p)\to H^*(X(\bC);\bZ/p)$ when $ch(k)=0$. 
In particular, note for smooth $X$,
\[ Q_i|CH^*(X)/p=Q_i|H^{2*,*}(X;\bZ/p)=0.\]
(See $\S 2$ in [Pi-Ya] for details.)
We will prove the following theorem without using Feshbach
theorem (Lemma 3.1).
\begin{thm}  Let $ G$ be simply connected and $H^*(G)$ has $p$-torsion.  Then the restriction map
\[ \rho^*_{CH}: CH^2(BG)\to CH^2(BT)^W\]
is not surjective.
\end{thm}
\begin{proof}
(See  $\S 2,3$ in [Pi-Ya].)
At first, we note that $H^*(BT)^W\cong CH^*(BT)^W$
since $H^*(BT  )\cong CH^*(BT)$.
If $H^*(G)$ has $p$-torsion, then $G$ has a subgroup isomorphic to $G_2$ (resp. $F_4$, $E_8$) for $p=2$
(resp. $p=3,5$).  
(For details, see [Ya2] or $\S 3$ in [Pi-Ya].)
We prove the theorem for $p=2$ but the other cases are proved similarly.  

 It is known that the inclusion $G_2\subset G$ induces
 a surjection $H^4(BG)\to  H^4(BG_2)\cong
\bZ_{(2)}$ and let us write by $x_4$ its generator.  
Then it is also known $Q_1x_4\not =0$ in $H^*(BG_2;\bZ/2)$
where $Q_1$ is the Milnor operation. 
Therefore $x_4\in H^4(BG_2)$ is not in the image of the cycle map
\[  cl: CH^2(BG_{2})\to H^4(BG_2).\]

On the other hand,  the element  $2x_4$ is in $Im(cl)$ because
it is represented by the second Chern class $c_2$.
Since $\rho^*\otimes \bQ$ is an isomorphism, 
$\rho_H^*(x_4)\not =0$.  But $\rho_H^*(x_4)$ is not
in the image $\rho^*_{CH}$.
\end{proof}

{\bf Remark.} The condition of simply connected is necessary.
By Vistoli ([Vi], [Ka-Mi]), it is known that $\rho^*_{CH}$ is surjective
for $G=PGL(p)$.

{\bf Remark.}
The above theorem is also proved by seeing
that $x_4$ is not generated by Chern classes, since
$CH^2(X)$ is always generated by Chern classes [To2].

Recall that for a smooth projective complex variety $X$,
the  integral Hodge conjecture is that the cycle map
\[ cl_{/Tor}:CH^*(X)\to H^{2*}(X)/Tor\cap H^{*,*}(X)\]
is surjective where $H^{*,*}(X)\subset H^{2*}(X;\bC)$ is the
submodule generated by $(*,*)$-forms.
Since $px_4=c_2$ in the proof of the above theorem
and $c_2\in H^{*,*}(X)$, we see $x_4\in H^{*,*}(X)$.

We know [To1], [Pi-Ya] that $B\bG_m\times BG$ can be approximated by smooth projective varieties.  Hence 
counterexamples for the integral Hodge conjecture
with $X=B\bG_m\times BG$ give the  examples
such that  $\rho_{CH}^*$ is not surjective.
\begin{lemma} Let $1\otimes y\not \in Im(cl_{/Tor})
\subset H^*(B\bG_m\times BG)$
be a counterexample of the integral Hodge conjecture.
Then it gives an example such that $\rho^*_{CH}$
is not surjective, namely, $\rho_H^*(y)\not \in Im(\rho_{CH}^*)$.
\end{lemma}
\begin{proof}
First note that $\rho_{H/Tor }^*:H^*(BG)/Tor
\to H^*(BT)^W$ is injective.
Since $CH^*(BT)^W\cong H^*(BT)^W$, we note
$\rho_{CH}^*=\rho_{H/Tor}^*cl_{/Tor}.$
Therefore $y\not \in Im(cl_{/Tor})$ implies that
$\rho_{H}^*(y)\not \in Im(\rho^*_{H/Tor}cl_{/Tor})
=Im(\rho_{CH}^*)$. 
\end{proof}

For each prime $p$, there are counterexamples $X=B\bG_m\times BG$ for the integral Hodge conjecture,
while they are not simply connected.
Indeed,  Kameko, 
Antineau and  Tripaphy ([Ka1,2], [An], [Tr]) show this for  $G=(SL_p\times SL_p)/\bZ/p$ and 
$SU(p^2)/\bZ/p$.  Hence they give the examples such that $\rho^*_{CH}$ are not surjective for non simply connected
and all  $p$ cases.
They proved these facts by using
Chern classes.

We also note its converse.
Recall [Pi-Ya] that the integral Tate conjecture over 
a finite field $k$ is the $ch(k)>0$ version of the integral Hodge conjecture.

\begin{lemma}
Let $x\in H^*(BT)^W$ such that $x\not \in \rho_{CH}^*$
but $x=\rho_{H}^*(y)$.  Moreover let $p^sy$ be represented by
Chern class for $s\ge 1$.  Then $1\otimes y\in H^*(B\bG_m\times BG)$
 gives a counterexample of the integral Hodge conjecture. 
It also gives a counterexample of the integral Tate conjecture for a finite field $k$ of $ch(k)\not =p$
\end{lemma}
\begin{proof}
Since $p^sy$ is represented by a Chern class, we see
$p^sy\in Im(cl)$.  Hence it is contained in $H^{*,*}(B\bG_m\times BG)$.  Hence so is $y$.  
Since $x\not \in \rho_{CH}^*$, we see $y\not \in cl_{/Tor}$.
For arguments for the integral Tate conjecture see [Pi-Ya].
\end{proof}

\section{$Spin(n)$ for $p=2$}

In this section, we study Chow rings for the cases $G=Spin(n)$,
$p=2$.
Recall that the $mod(2)$ cohomology is given by Quillen [Qu1]
\[ H^*(BSpin(n);\bZ/2)\cong \bZ/2[w_2,...,w_n]/J\otimes
\bZ/2[e]  \]
where $e=w_{2^h}(\Delta)$ and $J=(w_2,Q_0w_2,...,Q_{h-2}w_2)$.  Here $w_i$ is the Stiefel-Whitney class for the natural covering
$Spin(n)\to SO(n)$. The number $2^h$ is the Radon-Hurewitz number, dimension of the spin representation $\Delta$
(which is the representation $\Delta|C\not =0$ for the center
$C\cong \bZ/2\subset Spin(n)$).  The element $e$ is the Stiefel-Whitney class $w_{2^h}$ of the spin representation $\Delta$.

Hereafter this section we always assume
$G=Spin(n)$ and $p=2$.

By Kono [Ko], it is known that $H^*(BG;\bZ)$ has no
higher $2$-torsion,  that is
\[ H(H^*(BG;\bZ/2);Q_0)\cong (H^*(BG)/Tor)\otimes \bZ/2\]
where $H(A;Q_0)$ is the homology of $A$ with the differential
$d=Q_0$. 

For ease of arguments, let $n$ be odd i.e., $n=2k+1$. 
Let $T'$ be a maximal Torus of $SO(n)$ and $W'=W_{SO(n)}(T')$
its Weyl group.
Then $W'\cong S_k^{\pm}$ is generated by permutations and change of signs so that $|S_k^{\pm}|=2^kk!$.
Hence 
we have
\[H^*(BT')^{W'}\cong \bZ_{(2)}[p_1,...,p_k]\subset H^*(BT')\cong \bZ_{(2)}[t_1,...,t_k],\ |t_i|=2 \]
where the Pontriyagin class $p_i$ is defined by
$\Pi_i(1+ t_i^2)=\sum_ip_i$. 

For the projection $\pi:Spin(n)\to SO(n)$, the maximal torus
of $Spin(n)$ is given $\pi^{-1}(T')$ and $W=W_{Spin(n)}(T)\cong
W'$.   To seek the invariant $H^*(BT)^W$ is not so easy,
since the action $W\cong S_k^{\pm} $ is not given by
permutations and change of signs. 
Benson and Wood decided the $H^*(BT')^{W'}$-algebra
structure of  $H^*(BT)^{W}$ (Theorem 7.1 in [Be-Wo])
and proved 
\begin{thm} (Benson-Wood)
Let $G=Spin(n)$ and $p=2$.
Then $\rho^*_{H}$ is not surjective
if and only if $n\ge 11$ and  $n=3,4,5\ mod(8)$ (i.e., the quaternion case).
\end{thm}

Hereafter to study the Chow ring version, we assume $Spin(n)$
is in real case [Qu1], that is $n=8\ell-1, 8\ell, 8\ell+1$
(hence $\rho^*$ is surjective and $h=4\ell-1,4\ell-1,4\ell$ respectively). 

In this case,
it is known [Qu1] that each maximal elementary abelian $2$-group $A$ has 
$rank_2A=h+1$ and
\[ e|A=\Pi _{x\in H^1(B\bar A;\bZ/2)}(z+x)
\] where we identify 
$A\cong C\oplus \bar A$ and
$H^1(B\bar A;\bZ/2)\cong \bZ/2\{x_1,...,x_h\}$ is the $\bZ/2$ vector space generated by
$x_1,...,x_h$, and 
\[ H^*(BC;\bZ/2)\cong \bZ/2[z], \quad H^*(B\bar A;\bZ/2)\cong \bZ/2[x_1,...,x_h].\] 
The Dickson algebra is written as a polynomial algebra
\[ \bZ/2[x_1,....,x_h]^{GL_h(\bZ/2)}
\cong \bZ/2[d_0,....,d_{h-1}].\]
where $d_i$ is defined as
\[ e|A=z^{2^h}+d_{h-1}z^{2^{h-1}}+...+d_0z.\]
Hence we can [Qu1] also identify $d_i=w_{2^h-2^i}(\Delta)
\in H^*(BSpin(n);\bZ/2)$.
\begin{lemma}  (Lemma 2.1 in [Sc-Ya])
Milnor operations act on $d_i$ by
\[Q_{h-1}d_i=d_0d_i,\quad Q_{j-1}d_j=d_0,\ for\ 1 \le j,\]
\[Q_id_j=0\quad  for\ i<n-1\ and \ i\not =j-1.\]
\end{lemma}

\begin{lemma} (Corollary 2.1 in [Sc-Ya])
We have
\[Q_{h-1}e=d_0e\quad  and\quad Q_ke=0\ \  for\ 0\le k\le h-2.\]
\end{lemma}
\begin{thm}
Let $T\subset G=Spin(n)$ for $n=8\ell, 8\ell{\pm}1$.
There is an $e'\in CH^*(BT)^W$ such that
$e'\not \in Im(\rho_{CH}^*)$ and $\rho_H^*(e)=e'$ $mod(2)$.
\end{thm}
\begin{proof}
First note that $e|C=z^{2^h}$, which is not in the $Q_0$-image,
and hence $e$ itself is not.  From the preceding Lemma 4.2
we see $Q_0e=0$.  By Kono's result, we see
\[ 0\not =e\in H(H^*(BG;\bZ/2);Q_0)
  \cong  (H^*(BG)/Tor)\otimes \bZ/2).\]
Take  $e''\in H^*(BG)/Tor$ with that $e''=e$ $mod(2)$.
Then
\[ e'=\rho_H^*(e'')\not =0\ \ in\ \  H^*(BG)/Tor\subset 
H^*(BT)^W.\]

From the preceding Lemma 4.2, $Q_{h-1}(e)\not =0$.
hence  we see $e'\not \in \rho_{CH}^*$
by the existence of $Q_i$ in the motivic cohomology by
Voevodsky.
\end{proof}

Let $\Delta_{\bC}$ be the complex representation
induced from the real representation $\Delta$.  Then we see
(see Theorem 4.2 in [Sc-Ya])
\[ c_{2^{h-1}}(\Delta_{\bC})|C=2w_{2^h}|C=2z^{2^h}.\]
Of course this element $c_{2^{h-1}}(\Delta_{\bC})$ is in the
Chow ring $CH^*(BG)$.  hence we see that we can take
$2e'\in Im(\rho_{CH}^*)$.  

From the result by Benson-Wood,
we know $\rho^*_{H}$ is surjective in this (real) case.
Hence from Lemma 3.5 (or $Q_{h-1}(e)\not =0$), we have
\begin{cor}
Let $X=B\bG_m\times BSpin(n)$ with $n=8\ell, 8\ell\pm 1$
The element 
$1\otimes e\in H^{2^h}(X)\cap H^{2^{h-1},2^{h-1}}(X)$
 gives a counterexample
for the integral Hodge and the integral Tate conjectures, namely
$1\otimes e\not \in Im(cl_{H/Tor}$).
\end{cor}

\section{cobordism}

Let $BP^*(X)$ be the Brown-Peterson cohomology theory
with the coefficients ring  $BP^*=\bZ_{(p)}[v_1,v_2,...]$ of degree
$ |v_i|=-2(p^i-1)$.  Let $\Omega^*(X)=MGL^{2*,*}(X)
\otimes_{MU^*}BP^*$ be the $BP^*$-version of the algebraic cobordism ([Vo1], [Le-Mo1,2]) such that
$ \Omega^*(X)\otimes _{BP^*}\bZ_{(p)}\cong CH^*(X).$

We consider the cobordism version of the map $\rho_H^*$
\[ \rho_{\Omega}^*:\Omega^*(BG)\to
\Omega^*(BT)^W\cong  BP^*(BT)^{W}.\]
Although  $\bA^1$-homotopy category has the Becker-Gottlieb transfer
(this fact is announced in [Ca-Jo]),
we see
\[\tau\cdot \rho_{\Omega}^*=\chi(G/T)\ mod(v_1,v_2...)\]
which is not $\chi(G/T)$ in general.
So we can not have 
%$\Omega^*(BG)/Tor \subset \Omega^*(BT)^W.$
the $\Omega^*$-version of Feshbach's theorem.

We are interesting in an element $x\in \Omega^*(BT)^W$ which 
is in $Im(\rho_{\Omega}^*)$.  Of course, it is torsion free
in $\Omega^*(X)$, but there is a case such  that 
\[   0\not =x\in CH^*(BG)/p\cong \Omega^*(BG)\otimes_{BP^*}\bZ/p\]
and $x$ is $p$-torsion in $CH^*(BG).$

\begin{lemma}  
Let $f\in H^*(BT)^W$,  $f\not =0$ $mod(p)$, and identify $f\in gr\Omega^*(BT)\cong \Omega^*\otimes H^*(BT)$.
Let $f\not \in Im(\rho^*_{\Omega})$ but $v_mf\in
Im(\rho_{\Omega}^*)$.  Then $v_jf\in Im(\rho_{\Omega}^*)$
for all $0\le j\le m$.
Namely, there is $c_j\in \Omega^*(BG)$ such that
$\rho_{\Omega}^*(c_j)=v_jf,$
\[  c_j\not =0\in \Omega^*(BG)\otimes_{BP^*}\bZ/p\cong
   CH^*(BG)/p,\]
moreover $pc_j=0$ in $CH^*(BG)$ for $j>0$.
\end{lemma}
\begin{proof}
We consider the  Landweber-Novikov  cohomology operation $r_{a}$
in $gr\Omega^*(BT)\cong \Omega^*\otimes H^*(BT).$  By Cartan formula,
\[ r_a(v_mf)=\sum_{a=a'+a''}r_{a'}(v_m)r_{a''}(f).\]
Here  $r_{a''}(f)=0$ for $|a''|>0$ in $gr \Omega^*(BT)\cong \Omega^*\otimes H^*(BT)$.  We have operations $r_{\beta_j}(v_m)=v_j$ for $j\le m$, and we have the first statement.

  From the assumption, $f$ itself is not in the cycle map $\rho_{\Omega^*}$.
 Hence $v_jf$ is a $BP^*$-module generator in $\Omega^*(BT)^W\cap Im(\Omega^*(BG)))$.
Hence it is also non zero in $CH^*(BG)/p$.
Since $pv_jf=v_jpf\in v_jIm(\Omega^*(BG))$,
we have $pc_j=0\in CH^*(BG)$.
\end{proof}

%We can not prove the first assumption in the above theorem
%for $G=Spin(n)$, but we give a similar assumption.
%
%We recall the algebraic cobordism of $BC=B\bZ/2$
%\[ \Omega^*(BC)\cong BP^*(BC)\cong BP^*[[u]]/[2](u)\]
%where $u=z^2$ in $H^2(BC;\bZ/2)$ 
%and $[2](u)=u+_{F}u=2u+v_1u^2+...$
%the sum of the formal product for $BP$-theory.
%\begin{cor}  Let $G=Spin(n)$.  Let $x\in \Omega^*(BG)$ with
%$x|C=v_mu^{h'}$ $mod(u^{h'+1})$ where $h'=2^{h-1}$
% for some $m\ge 1$.  Then there are elements
%$c_i\in \Omega^*(BG)$ for $0\le i\le m$ with $c_i|C=v_iu^{h'}$
%mod $u^{h'+1}$  such that
%$ c_i\not =0\in  CH^*(BG)/2,$ and $2c_j=0$ in $CH^*(BG)$ %for $j>0$.
%\end{cor}

We consider the Atiyah-Hirzebruch spectral sequence
(AHss)
\[ E_2^{*,*'}\cong H^*(X;BP^{*'})\Longrightarrow BP^*(X)\]
It is known that
\[ (*)\quad d_{2p^i-1}(x)=v_i\otimes Q_i(x)\quad mod(p,v_1,...,v_{i-1}).\]
In general, there are many other types of non zero differential.
However we consider cases that differentials are only of 
this form.

\begin{lemma} Let $X=BSpin(n)$
and $n=8\ell,8\ell\pm1$.  In AHss for $BP^*(X)$,
assume all non zero differentials are of form $(*)$.  Then
$pe,v_1e,...,v_{h-2}e$ are all permanent cycles.
\end{lemma}
\begin{proof} We use Lemma 4.2 in the preceding section.  
First recall $Q_i(d_0)=0,\ Q_i(e)=0$ for $i<h-1$.
Therefore $d_0e$ exists in $E_{2^{h}-1}$.

Since $Q_{j-1}d_j=d_n$ and $Q_k(d_j)=0$ for $k<j-1$,
the differential in AHss is
\[ d_{2^j-1}(d_je)=v_{j-1}\otimes Q_{j-1}(d_je)=v_{j-1}d_0e.\]
Hence we have  
$  (2,v_1,v_2,...,v_{h-2})(d_0e)=0$ in $E_{2^{h}-1}^{*,*'}.$

Now we study the differential 
\[d_{2^{h}-1}(e)=v_{h-1}Q_{h-1}(e)=v_{h-1}d_0e.\]
Note that $e$ is $BP^*$-free in $E_{2^{h}-1}^{*,,*'}$,
since $e|C=z^{2^h}$ and $e\not \in Im(Q_i)$.
Hence we have
\[ Ker(d_{2^{h}-1})\cap BP^*\{e\}
\cong Ideal(2,v_1,...,v_{h-2})\{e\}.\]
(In this paper, $R\{a,b,...\}$ means the $R$-free module
generated by $a,b,...$)
By the assumption $(*)$ for differentials, $pe$,$v_1e$,...,
$v_{h-2}e$ are all permanent cycles.
\end{proof}
For  $7\le n\le 9$, AHss converging $BP^*(BSpin(n))$ is computed in [Ko-Ya], ([Sc-Ya] also),
 and it is known that $(*)$ is satisfied.
\begin{cor}
For $n=7,8$ (resp. $n=9$), the elements $pe,v_1e$
(resp. $pe,v_1e,v_2e$) are in $Im(\rho^*_{BP})\subset BP^*(BT)^W$ ( but $e$ itself is not).
\end{cor}

Let $K(n)^*(X)$ be the  Morava $K$-theory
with the coefficients ring $K(n)^*\cong
\bZ/p[v_n,v_n^{-1}]$, and 
$AK(n)^*(X)=AK(n)^{2*,*}(X)$ its algebraic version.
Here we consider  an assumption such that
\[ (**)\ \ AK(n)^*(BG)\to K(n)^*(BG)\quad is\ surjecive.\]
It is known by Merkurjev (see [To1] for details) that $AK^*(BG)\cong K^*(BG)$
for the algebraic $K$-theory $AK^*(X)$ and the
complex $K$-theory $K^*(X)$, which induces 
$AK(1)^*(BG)\cong K(1)^*(BG)$.  Hence $(**)$ is correct
when  $n=1$ for all $G$.
\begin{lemma}  Let $X=BSpin(n),\ n=8\ell, 8\ell\pm 1$ and suppose $(*), (**)$.
Moreover let $h\ge 3$.
Then $v_{h-2}e\in Im(\rho^*_{\Omega})$, and hence
there is $c_i\in CH^*(X)$ for $0\le i\le h-2$ in Lemma 5.1.
\end{lemma}
\begin{proof}
First note $0\not =v_{h-2}e\in K(h-2)^*(X)$
(hence so is $e$).  On the other hand
\[ AK(h-2)^*(X)\cong K(h-2)^*\otimes CH^*(X)/I\]
for some ideal $I$ of $CH^*(X)$.  Therefore there is
an element $c\in CH^*(X)$ which corresponds $v_{h-2}^se$
that is $cl_{\Omega}(c)=v_{h-2}^se$ for
$cl_{\Omega}: \Omega^*(X)\to BP^*(X).$
Since $e\not \in Im(cl_{\Omega})$, we see $s$ must be positive.
The possibility  of
\[ |v_{h-2}^se|=-2(2^{h-2}-1)s+2^h>0\]
is $s=1$ or $s=2$.  Note $|v_{h-2}^2e|=4$.
However it is known by Totaro (Theorem 15.1 in[To2]),
\[ cl : CH^2(X)\to H^4(X)\quad is \ injective.
\]
Hence $s=1$ and $cl_{\Omega}(c)=v_{h-1}e$.
\end{proof}

\begin{cor}
For $X=BSpin(7)$, there is an element $c\in CH^3(X)$
such that $c\not =0\in CH^*(X)/2$, $cl(c)=0$ and
$\rho^*_{\Omega}(c)\not =0\in \Omega^*(BT)^W$.
\end{cor}

\section{ $Spin(7)$ for $p=2$}

 Let $G$ be a compact
Lie group.  Consider the restriction map
\[ res_{H/p}: H^*(BG;\bZ/p)\to Lim_{V:el.ab.}H^*(BV;\bZ/p)^{W_G(A)}\]
where $W_G(A)=N_G(A)/C_G(A)$ and $V$ ranges in the conjugacy classes of maximal
elementary abelian $p$-groups.  Quillen [Qu2] showed this
$res_{H/p}$ is an $F$-isomorphism (i.e. its kernel and cokernel
are generated by nilpotent elements).  We consider its integral
version
\[ res_{H} : H^*(BG)\to \Pi_{A:ab.}H^*(BA)^{W_G(A)},\]
where $A$ ranges in the conjugacy classes of maximal
abelian subgroups of $G$.

Hereafter this section, we assume $G=Spin(7)$ and $p=2$
and hence $h=3$.
The number of conjugacy classes of the maximal abelian subgroups
of $G$ is two, one is the torus $T$ and the other is
$A'\cong (\bZ/2)^4$ which is not contained in $T$.
The Weyl group is $W_G(A')\cong \la U,GL_3(\bZ/2)\ra\subset GL_4(\bZ/2)$ where $U$ is the maximal unipotent group in
$GL_4(\bZ/2)$.
It is well known
\[ H^*(BG;\bZ/2)\cong H^*(BA';\bZ/2)^{W_G(A')}\]
\[ \cong \bZ/2[w_4,w_6,w_7,w_8],\quad Q_0w_6=w_7\]
where $w_i$  for $i\le 7$ (resp. $i=8$)
are the Stiefel-Whitney class for the representation
induced from $Spin(7)\to SO(7)$
(resp. the spin representation $\Delta$
and hence $w_8=w_8(\Delta)=e$).

Since $H^*(BG)$ has just $2$-torsion by Kono, and hence 
the restriction map $res_{H}$ injects $Tor\subset H^*(BG)$ into $H^*(BA';\bZ/2)^{W_G(A')}$.

Next we consider the integral case.
 Also note $H^*(BG)$ has just $2$-torsion and
\[ (H^*(BG)/Tor)\otimes \bZ/2\cong H(H^*(BG;\bZ/2);Q_0).\]
 Since $Q_0w_i=0$ for $i\not =6$ and $Q_0w_6=w_7$, we have
\[ H(H^*(BG;\bZ/2);Q_0)\cong \bZ/2[w_4,c_6,w_8]\quad c_6=w_6^2.\]
Of course the right hand side ring has no nonzero nilpotent elements.  Hence we 
see that $\rho_H^*$ is surjective
and 
\[H^*(BT)^W\otimes \bZ/2\cong \bZ/2[w_4,c_6,w_8].\]
The integral cohomogy is written as
\[ H^*(BG)\cong \bZ_{(2)}[w_4,c_6,w_8]\otimes(\bZ_{(2)}\{1\}
\oplus \bZ/2[w_7]\{w_7\}).\]
In particular, we note $res_H$ is injective.

Next we consider the Atiyah-Hirzebruch spectral
sequence
\[ E_2^{*,*'}\cong H^*(BG)\otimes BP^*
\Longrightarrow BP^*(BG).\]
Its differentials have forms of $(*)$ in $\S 5$.
Using $Q_1(w_4)=x_7, Q_1(e_7)=c_7$,  $Q_2(w_8)=w_7w_8$
and $Q_3(w_7w_8)=c_7c_8,$
we can compute the spectral sequence
(while it is some what complicated)
\[ grBP^*(BG)\cong BP^*[c_4,c_6,c_8]
\{1,2w_4,2w_8,2w_4w_8,v_1w_8\}\]
\[ \oplus BP^*/(2,v_1,v_2)[c_4,c_6,c_7,c_8]\{c_7\}/(v_3c_7c_8).\]
Hence $BP^*(BG)\otimes_{BP^*}\bZ_{(2)}$
is isomorphic to 
\[ \bZ_{(2)}^*[c_4,c_6,c_8]\{1,2w_4,2w_8,2w_4w_8,v_1w_8\}/(2v_1w_8)\]
\[ \oplus \bZ/2[c_4,c_6,c_7,c_8]\{c_7\}.\]

On the other hand, the Chow ring of $BG$ is given
by Guillot [Gu,Ya]
\[ CH^*(BG)\cong BP^*(BG)\otimes_{BP^*}\bZ_{(2)}\]
\[ \cong \bZ_{(2)}[c_4,c_6,c_8]
\otimes(\bZ_{(2)}\{1,c_2',c_4'.c_6'\}
\oplus \bZ/2\{\xi_3\}\oplus \bZ/2[c_7]\{c_7\})\]
where $cl(c_i)=w_i^2$, $cl(c_2')=2w_4$, $cl(c_4')=2w_8$, $cl(c_6')=2w_4w_8$,  and $cl(\xi_3)=0$, $|\xi_3|=6$.
Note $cl_{\Omega}(\xi_3)=v_1w_8$ in $BP^*(BT)^W$,
and $\xi_3=c$ in Corollary 5.5.
 Hence we have
\[ CH^*(BG)/Tor\cong  \bZ_{(2)}[c_4,c_6,c_8]
\{1,c_2',c_4'.c_6'\}\]
\[ \qquad \qquad \subset 
\bZ_{(2)}[w_4,c_6,w_8]\cong CH^*(BT)^W.\]
In fact the nilpotent ideal in $(CH^*(BG)/(Tor))\otimes \bZ/2$ is generated by $c_2',c_4',c_6'$.

Next we consider the Chow rings version for the restriction map
\[ res_{CH} : CH^*(BG)\to \Pi_{A:ab.}CH^*(BA)^{W_G(A)}.\]
Recall $CH^*(BA')\cong \bZ_{(2)}[y_1,...,y_4]$ with $cl(y_i)=x_i^2$.
Hence we have
\[ (CH^*(BA')/2)^{W_G(A')}\cong \bZ/2[c_4,c_6,c_7,c_8].\]
Since $Tor$ is just $2$-torsion, we have
\begin{lemma} For the torsion ideal $Tor$ in $CH^*(BG)$
\[  res_{CH}(Tor) \cong \bZ/2[c_4,c_6,c_8,c_7]\{c_7\}    
       \subset CH^*(BA').\]
\end{lemma}
Thus we see that
$ Ker(res_{CH})\cong \bZ/2[c_4,c_6,c_8]\{\xi_3\},$
which is the ideal of Griffiths elements.
We write down the above results.
\begin{thm}  Let $(G,p)=(Spin(7),2)$. Let $Grif$ 
be the ideal generated by Griffiths elements and
$D=\bZ_{(2)}[c_4,c_6,c_8]$.  Then we have
\[ CH^*(BG)/Tor\cong D\{1,2w_4,2w_8,2w_4w_8\}\]
\[ \qquad \qquad \qquad \qquad \subset
D\{1,w_4,w_8,w_4w_8\}\cong CH^*(BT)^W,\ \  with \ w_i^2=c_i,\]
\[Tor/Grif\cong D/2[c_7]\{c_7\},\quad Grif\cong D/2\{\xi_3\}.\]
\end{thm}

There is only one nonzero element $\xi_3$ in $CH^3(BG)/2$, and this $\xi_3=c$ in Corollary 5.5 in the preceding section.
Thus we see Theorem 1.2 in the introduction.
\begin{cor}
Take an element $\xi\in \Omega^*(BG)$ such that
$\xi=\xi_3$ in $\Omega^*(BG)\otimes _{BP^*}\bZ_{(2)}
\cong CH^*(BG)$.  We also identify $c_i\in \Omega^*(BG)$. Then
\[ \bZ/2[c_4.c_6,c_8]\{\xi\}\subset \Omega^*(BT)^W/2.\]
\end{cor}

\begin{thm}
Let $G=Spin(7)$ and $p=2$.  We consider a (non natural) map
\[ CH^*(BG)\subset \Omega^*(BG)\to \Omega^*(BT)\]
(here the first inclusion is just one section of the projection
$\Omega^*(BG)\to CH^*(BG)$).
 Then the following restriction map is injective
\[res :  CH^*(BG)\subset CH^*(BA')\times \Omega ^*(BT).\]
\end{thm}

\begin{cor}
Let $J\subset BP^*$ is the ideal $J=(2^2,2v_1,v_1^2,v_2,...)$
so that $BP^*/J \cong \bZ/4\{1\}\oplus \bZ/2\{v_1\}.$  Let $D=\bZ_{(2)}[c_4,c_6,c_8]$.  Then
\[ \Omega^*(BG)/J
\cong (\Omega^*/J\otimes D\{1,c_2',c_4',c_6',\xi_3\}/(2\xi_3=v_1c_4'))
 \oplus D/2[c_7]\{c_7\}.\]
\end{cor}

\section{The exceptional group $F_4$, $p=3$.}

In this section, we assume $(G,p)=(F_4,3)$.
(However similar arguments also work for
$(G,p)=(E_6,3),(E_7,3)$ and $(E_8,5)$, [Ka-Ya].)
Toda computed the $mod(3)$ cohomology of $BF_4$.
(For details see [Tod].)
\[ H^*(BG;\bZ/3)\cong C\otimes D,\quad where\]
\[   C=F\{1,x_{20},x_{20}^2\}\oplus \bZ/3[x_{26}]\otimes \Lambda(x_9)\otimes \{1,x_{20},x_{21},x_{26}\}\]
\[  D=\bZ_{(3)}[x_{36},x_{48}],\quad F=\bZ_{(3)}[x_4,x_8]. \]

Using that $H^*(BG)$ has no higher $3$-torsion and
$Q_0x_8=x_9$, $Q_0x_{20}=x_{21}$, $Q_0x_{25}=x_{26}$,
we can compute
\[ H^*(BG)\cong D\otimes C'\quad where \]
\[ C'/Tor \cong Z_{(3)}\{x_4\}\oplus E,\quad where\ E=F\{ab|a,b\in \{x_4,x_8,x_{20}\}\}\]
\[ C'\supset Tor\cong \bZ/3[x_{26}]\{x_{26},x_{21}
,x_9, x_9x_{21}\}.\]

Note $x_{26}=Q_2Q_1(x_4)$ in Theorem 2.2 and 
\[ H^*(BT;\bZ/p)^W\cong H^{even}(BG;\bZ/3)/(Q_2Q_1x_4)
\cong D\otimes F\{1,x_{20},x_{20}^2\}.\]
Hence we see
\[ (H^*(BG)/Tor)\otimes \bZ/3\cong
D/3\otimes (\bZ/3\{1,x_4\}\oplus E)\subset
D/3\otimes F\{1,x_{20},x_{20}^2\}.\]
Thus from Lemma 2.3,  we see $\rho_H^*$ is surjective
and 
\[ H^*(BT)^W\cong H^*(BG)/Tor\cong D\otimes(\bZ_{(3)}\{1,x_4\}\oplus E).\]

%{\bf Remark.}
%From Toda (relation $r_{15}$ in page 98 in [Tod]) it is known
%\[ r_{15}=y_{20}^3+x_8^2x_4y_{20}^2-x_{48}x_{4}^3-x_{36}x_8^3=0.\]
%Hence $(H^*(BG)/Tor)\otimes \bZ/3\subset 
% F/3\{1,x_{20},x_{20}^2\}$ has no nilpotent
%element.
%\bf Remark.} Elements  $x_8,x_{20}\in (H^*(BT)/3)^W$
%can not be  extended  as elements in 
%the integral invariants $H^*(BG)^W$.

Next we consider the Atiyah-Hirzebruch spectral
sequence [Ko-Ya]
\[ E_2^{*,*'}\cong H^*(BG)\otimes BP^*
\Longrightarrow BP^*(BG).\]
Its differentials have forms of $(*)$ in $\S 5$.
Using $Q_1(x_4)=x_9, Q_1(x_{20})=x_{25}, Q_1(x_{21})=x_{26}$ and $Q_2x_{9}=x_{26}$ we can compute
\[ grBP^*(BG)\cong D\otimes(BP^*\otimes(\bZ_{(3)}\{1,3x_4\}\oplus E)
 \oplus BP^*/(3,v_1,v_2)[x_{26}]\{x_{26}\}).\]
Hence we have
\[BP^*(BG)\otimes_{BP^*}\bZ_{(3)}\cong D\otimes
    (\bZ_{(3)}\{1,3x_4\}\oplus E\oplus \bZ/3[x_{26}]\{x_{26}\}).\]
\begin{prop}
Let $(G,p)=(F_4,3)$ and $Tor \supset Grif$ be the ideal generated by Griffiths elements.  Then we have
\[ CH^*(BG)/Tor\subset   D\otimes (\bZ_{(3)}\{1,3x_4\}\oplus E)\subset H^*(BG)/Tor,\]
\[ Tor/Grif \cong D\otimes \bZ/3[x_{26}]\{x_{26}\}.\]
\end{prop}

If Totaro's conjecture is correct, then $Grif=\{0\}$
and the first inclusion is an isomorphism.
From [Ya1], it is known that if $x_8^2\in Im(cl)$ for the cycle map $cl$, then we can show that $cl$ itself
is surjective. However it seems still unknown 
whether $x_8^2\in Im(cl)$ or not.  
\begin{cor}
Let $(G,p)=(F_4,3)$.
If $(**)$ in $\S 5$ is correct for some $n\ge 2$, then
the cycle map $CH^*(BG)\to BP^*(BG)\otimes_{BP^*}\bZ_{(3)}$
 is surjective and
\[ CH^*(BG)/Tor\cong  D\otimes (\bZ_{(3)}\{1,3x_4\}\oplus E).\]
\end{cor}
\begin{proof}
The corollary follows from 
$|v_nx_8^2|=16-2(3^n-1)\le 0$.
\end{proof}

\end{document}